\documentclass[11pt,leqno]{article}
\usepackage[margin=1in]{geometry} 
\usepackage{amssymb,amsfonts,amsmath,bbm,mathrsfs,stmaryrd,mathtools}
\usepackage{xcolor}
\usepackage{url}

\usepackage{graphicx}

\usepackage{accents}

\usepackage{extarrows}

\usepackage[shortlabels]{enumitem}
\usepackage{tensor}

\usepackage{xr}
\usepackage[T1]{fontenc}
\usepackage[utf8]{inputenc}

\usepackage[colorlinks,
linkcolor=black!75!red,
citecolor=blue,
pdftitle={},
pdfproducer={pdfLaTeX},
pdfpagemode=None,
bookmarksopen=true,
bookmarksnumbered=true,
backref=page]{hyperref}

\usepackage{tikz}
\usetikzlibrary{arrows,calc,decorations.pathreplacing,decorations.markings,decorations.shapes,intersections,shapes.geometric,through,fit,shapes.symbols,positioning,decorations.pathmorphing}

\makeatletter
\newlength\zig@L
\newlength\zig@La
\newlength\zig@Lb

\newcommand{\xzigrightarrow}[2][]{%
  \mathrel{%
    \settowidth{\zig@La}{$\scriptstyle #2$}%
    \settowidth{\zig@Lb}{$\scriptstyle #1$}%
    \zig@L=\zig@La\relax
    \ifdim\zig@Lb>\zig@L \zig@L=\zig@Lb\fi
    \advance\zig@L by 2.2em\relax
    \tikz[baseline=-0.65ex]{%
      \draw[->,
            line cap=round,
            decorate,
            decoration={zigzag,segment length=4pt,amplitude=1.1pt}]%
        (0,0) -- (\zig@L,0)
        node[midway,above=2pt] {$\scriptstyle #2$}%
        \if\relax\detokenize{#1}\relax\else
          node[midway,below=2pt] {$\scriptstyle #1$}%
        \fi
      ;
    }%
  }%
}
\makeatother

\makeatletter
\newcommand{\squigjoin}{1mu} 

\def\sqleft@{\sim}                    
\def\sqmid@{\sim\mkern-\squigjoin}    

\def\rightsquigarrowfill@{%
  \arrowfill@{\sqleft@}{\sqmid@}{\mkern-4mu\succ}%
}

\newcommand{\xrightsquigarrow}[2][]{%
  \ext@arrow 0359\rightsquigarrowfill@{#1}{#2}%
}
\makeatother

\makeatletter
\newcommand*\circled[1]{\tikz[baseline=(char.base)]{
    \node[shape=circle, draw, inner sep=0pt, 
    minimum height={\f@size},] (char) {\vphantom{WAH1g}#1};}}
\makeatother

\makeatletter
\DeclareRobustCommand\widecheck[1]{{\mathpalette\@widecheck{#1}}}
\def\@widecheck#1#2{%
    \setbox\z@\hbox{\m@th$#1#2$}%
    \setbox\tw@\hbox{\m@th$#1%
       \widehat{%
          \vrule\@width\z@\@height\ht\z@
          \vrule\@height\z@\@width\wd\z@}$}%
    \dp\tw@-\ht\z@
    \@tempdima\ht\z@ \advance\@tempdima2\ht\tw@ \divide\@tempdima\thr@@
    \setbox\tw@\hbox{%
       \raise\@tempdima\hbox{\scalebox{1}[-1]{\lower\@tempdima\box
\tw@}}}%
    {\ooalign{\box\tw@ \cr \box\z@}}}
\makeatother

\usepackage{braket}

\usepackage[amsmath,thmmarks,hyperref]{ntheorem}
\usepackage{cleveref}

\newcommand\nthalias[1]{\AddToHook{env/#1/begin}{\crefalias{lemma}{#1}}}

\nthalias{definition}
\nthalias{example}
\nthalias{examples}
\nthalias{remark}
\nthalias{remarks}
\nthalias{convention}
\nthalias{notation}
\nthalias{construction}
\nthalias{sketch}
\nthalias{theoremN}
\nthalias{propositionN}
\nthalias{corollaryN}
\nthalias{lemma}
\nthalias{proposition}
\nthalias{corollary}
\nthalias{theorem}
\nthalias{conjecture}
\nthalias{question}
\nthalias{assumption}

\creflabelformat{enumi}{#2#1#3}

\crefname{section}{Section}{Sections}
\crefformat{section}{#2Section~#1#3} 
\Crefformat{section}{#2Section~#1#3} 

\crefname{subsection}{\S}{\S\S}
\AtBeginDocument{%
  \crefformat{subsection}{#2\S#1#3}%
  \Crefformat{subsection}{#2\S#1#3}%
}

\crefname{subsubsection}{\S}{\S\S}
\AtBeginDocument{%
  \crefformat{subsubsection}{#2\S#1#3}%
  \Crefformat{subsubsection}{#2\S#1#3}%
}

\theoremstyle{plain}

\newtheorem{lemma}{Lemma}[section]

\newtheorem{corollary}[lemma]{Corollary}
\newtheorem{theorem}[lemma]{Theorem}

\theoremstyle{plain}
\theoremnumbering{Alph}

\theoremstyle{plain}
\theorembodyfont{\upshape}
\theoremsymbol{\ensuremath{\blacklozenge}}

\newtheorem{example}[lemma]{Example}

\newtheorem{remark}[lemma]{Remark}
\newtheorem{remarks}[lemma]{Remarks}

\newtheorem{notation}[lemma]{Notation}
\newtheorem{construction}[lemma]{Construction}

\crefname{definition}{definition}{definitions}
\crefformat{definition}{#2definition~#1#3} 
\Crefformat{definition}{#2Definition~#1#3} 

\crefname{ex}{example}{examples}
\crefformat{example}{#2example~#1#3} 
\Crefformat{example}{#2Example~#1#3} 

\crefname{exs}{example}{examples}
\crefformat{examples}{#2example~#1#3} 
\Crefformat{examples}{#2Example~#1#3} 

\crefname{remark}{remark}{remarks}
\crefformat{remark}{#2remark~#1#3} 
\Crefformat{remark}{#2Remark~#1#3} 

\crefname{remarks}{remark}{remarks}
\crefformat{remarks}{#2remark~#1#3} 
\Crefformat{remarks}{#2Remark~#1#3} 

\crefname{convention}{convention}{conventions}
\crefformat{convention}{#2convention~#1#3} 
\Crefformat{convention}{#2Convention~#1#3} 

\crefname{notation}{notation}{notations}
\crefformat{notation}{#2notation~#1#3} 
\Crefformat{notation}{#2Notation~#1#3} 

\crefname{table}{table}{tables}
\crefformat{table}{#2table~#1#3} 
\Crefformat{table}{#2Table~#1#3}

\crefname{lemma}{lemma}{lemmas}
\crefformat{lemma}{#2lemma~#1#3} 
\Crefformat{lemma}{#2Lemma~#1#3} 

\crefname{proposition}{proposition}{propositions}
\crefformat{proposition}{#2proposition~#1#3} 
\Crefformat{proposition}{#2Proposition~#1#3} 

\crefname{propositionN}{proposition}{propositions}
\crefformat{propositionN}{#2proposition~#1#3} 
\Crefformat{propositionN}{#2Proposition~#1#3} 

\crefname{corollary}{corollary}{corollaries}
\crefformat{corollary}{#2corollary~#1#3} 
\Crefformat{corollary}{#2Corollary~#1#3} 

\crefname{corollaryN}{corollary}{corollaries}
\crefformat{corollaryN}{#2corollary~#1#3} 
\Crefformat{corollaryN}{#2Corollary~#1#3} 

\crefname{theorem}{theorem}{theorems}
\crefformat{theorem}{#2theorem~#1#3} 
\Crefformat{theorem}{#2Theorem~#1#3} 

\crefname{theoremN}{theorem}{theorems}
\crefformat{theoremN}{#2theorem~#1#3} 
\Crefformat{theoremN}{#2Theorem~#1#3} 

\crefname{enumi}{}{}
\crefformat{enumi}{#2#1#3}
\Crefformat{enumi}{#2#1#3}

\crefname{assumption}{assumption}{Assumptions}
\crefformat{assumption}{#2assumption~#1#3} 
\Crefformat{assumption}{#2Assumption~#1#3} 

\crefname{construction}{construction}{Constructions}
\crefformat{construction}{#2construction~#1#3} 
\Crefformat{construction}{#2Construction~#1#3} 

\crefname{sketch}{sketch}{Sketches}
\crefformat{sketch}{#2sketch~#1#3} 
\Crefformat{sketch}{#2Sketch~#1#3} 

\crefname{question}{question}{Questions}
\crefformat{question}{#2question~#1#3} 
\Crefformat{question}{#2Question~#1#3} 

\crefname{equation}{}{}
\crefformat{equation}{(#2#1#3)} 
\Crefformat{equation}{(#2#1#3)}

\numberwithin{equation}{section}

\theoremstyle{nonumberplain}
\theoremsymbol{\ensuremath{\blacksquare}}

\newtheorem{proof}{Proof}
\newcommand\pf[1]{\newtheorem{#1}{Proof of \Cref{#1}}}

\newcommand\bB{{\mathbb B}}

\newcommand\bG{{\mathbb G}}
\newcommand\bH{{\mathbb H}}

\newcommand\bK{{\mathbb K}}
\newcommand\bL{{\mathbb L}}

\newcommand\bP{{\mathbb P}}

\newcommand\bU{{\mathbb U}}

\newcommand\bZ{{\mathbb Z}}

\newcommand\cC{{\mathcal C}}
\newcommand\cD{{\mathcal D}}
\newcommand\cE{{\mathcal E}}
\newcommand\cF{{\mathcal F}}

\newcommand\cM{{\mathcal M}}

\newcommand\cO{{\mathcal O}}

\newcommand\cR{{\mathcal R}}

\newcommand\fa{{\mathfrak a}}
\newcommand\fb{{\mathfrak b}}

\newcommand\fe{{\mathfrak e}}

\newcommand\fg{{\mathfrak g}}
\newcommand\fh{{\mathfrak h}}

\newcommand\fl{{\mathfrak l}}

\newcommand\fp{{\mathfrak p}}
\newcommand\fq{{\mathfrak q}}

\newcommand\fu{{\mathfrak u}}

\newcommand\fz{{\mathfrak z}}

\DeclareMathOperator{\id}{id}

\DeclareMathOperator{\Hom}{\mathrm{Hom}}
\DeclareMathOperator{\Ext}{\mathrm{Ext}}
\DeclareMathOperator{\Aut}{\mathrm{Aut}}

\newcommand{\cat}[1]{\textsc{#1}}

\newcommand{\qedhere}{\mbox{}\hfill\ensuremath{\blacksquare}}

\newcommand{\comment}[1]{}

\newcommand{\xrightarrowdbl}[2][]{%
  \xrightarrow[#1]{#2}\mathrel{\mkern-14mu}\rightarrow
}

\title{Coincident Poisson structures on principal-bundle moduli spaces}
\author{Alexandru Chirvasitu}

\begin{document}

\date{}

\newcommand{\Addresses}{{
  \bigskip
  \footnotesize

  \textsc{Department of Mathematics, University at Buffalo}
  \par\nopagebreak
  \textsc{Buffalo, NY 14260-2900, USA}  
  \par\nopagebreak
  \textit{E-mail address}: \texttt{achirvas@buffalo.edu}

}}

\maketitle

\begin{abstract}
  Consider a reductive complex algebraic group $\mathbb{G}$ equipped with an action by a linearly reductive affine group scheme $\mathbb{K}$. The extension of $\mathfrak{p}^*$ by $\mathfrak{p}$ induced by an $(\mathbb{K},\mathfrak{g})$-invariant symmetric non-degenerate bilinear form on $\mathfrak{g}:=Lie(\mathbb{G})$, for a $\mathbb{K}$-invariant parabolic $\mathbb{P}\le \mathbb{G}$, is $\mathbb{K}$-equivariantly isomorphic to the extension obtained via the standard bialgebra structure attached to a $\mathbb{K}$-invariant Cartan/Borel pair $\mathbb{H}\le \mathbb{B}\le \mathbb{P}\le \mathbb{G}$ and the same bilinear form. Associating bundle extensions on an elliptic curve $E$ to said $\mathfrak{p}$-module extensions, this identifies Poisson structures on the smooth locus of the principal-$\mathbb{P}$-bundle moduli space over $E$ respectively defined by Balduzzi (using the former extension) and Feigin-Odesskii (via the standard bialgebra structure). This in particular verifies Feigin-Odesskii's identification of the bialgebra-induced symplectic leaves with the loci of bundles mutually isomorphic after forgetting structure along $\mathbb{P}\le \mathbb{G}$.
\end{abstract}


\vspace{.5cm}

\noindent{MSC 2020: 14H60; 17B56; 17B62; 53D17; 17B55; 53D30; 20G20; 14H52


}


\section*{Introduction}

Consider a connected reductive complex linear algebraic group $\bG$ with attached Lie algebra $\fg$. For a parabolic subgroup $\bP\le \bG$, \cite[\S 3]{MR2396472} equips the smooth locus of the moduli space $\cat{Bun}_E^{\bP}$ of principal $\bP$-bundles on an elliptic curve $E$ with a \emph{Poisson structure} (\cite[\S 1.2.1, Definition 1.5]{lgpv_poiss_2013}, \cite[Definition 2.7]{cfm_lec-poiss_2021}), construed there in several alternative ways.

On the other hand, Poisson structures on principal-bundle moduli spaces over elliptic curves as introduced in Feigin and Odesskii's work \cite{FO98} have received attention in a growing body of literature, of which \cite{MR5018405,MR4956460,MR5014131,HP1,hua2023elliptic,MR4855099,MR4214399} (with their references) will provide a sampling. 

\cite{MR2396472,FO98} both define the said Poisson structures by associating $\bP$-bundles on $E$ to $(\fp:=Lie(\bP))$-module extensions of the coadjoint representation $\fp^*$ by the adjoint $\fp$, with the difference lying in how those extensions are constructed (\Cref{con:bunp.poiss.no.bialg,con:bunp.poiss.bialg} below provide more expansive reminders):
\begin{itemize}[wide]
\item \cite[(3.2)]{MR2396472} pulls back $\fg$ itself, as an extension of $\fg/\fp$ by $\fp$, along the map $\fp^*\to \fg/\fp$ induced by a symmetric non-degenerate invariant bilinear form on $\fg$ fixed throughout; 

\item while the somewhat terse comment on \cite[p.67]{FO98} simply selects as the requisite 1-cocycle $\fp\to \fp^{\otimes 2}$ the restriction to $\fp$ of what the authors refer to as the \emph{standard bialgebra structure} on $\fg$. 
\end{itemize}
The latter phrase is amenable to some interpretation, but an examination of how it is employed in the Kac-Moody context of \cite[Example 1.3.8]{cp_qg} shows that what is needed in order to alight on such a structure are Cartan/Borel subalgebras $\fh\le \fb\le \fg$ and a symmetric non-degenerate ad-invariant bilinear form on $\fg$; hence the phrase \emph{$(\fb,\fh,\Braket{-\mid-})$-standard} featuring in the sequel. 

With a view towards potential ramifications pertinent to \emph{equivariant principal bundles} (in the sense, say, of \cite[\S I.8]{td_transf-gp} or \cite[Chapter VII]{may_eq-hotop_1996}), to be pursued elsewhere, it will be convenient to equivariantize the discussion by equipping all objects in sight with actions by a \emph{linearly reductive} \cite[Definition 1.4]{fkm} affine group scheme $\bK$: $\bG$ carries such an action with $\bB\le \bP\le \bG$ and $\Braket{-\mid-}$ $\bK$-invariant. Whatever auxiliary structures emerge as necessary (e.g. Levi factors, opposite parabolic/Borel subgroups, etc.) will be available $\bK$-equivariantly, not by assumption, but rather as a consequence of the assumed linear reductivity. 

\begin{theorem}\label{th:same.poiss}
  Let $\bK$ be a linearly reductive affine group $\Bbbk$-scheme over an algebraically closed characteristic-0 field and consider
  \begin{itemize}[wide]
  \item a $\bK$-equivariant chain of embeddings $\bB\le \bP\le \bG$ of Borel and parabolic subgroups into a reductive connected linear algebraic $\Bbbk$-group acted upon by $\bK$;
  \item and a $\bK$-$\fg$-invariant symmetric non-degenerate bilinear form $\Braket{-\mid -}$ on $\fg:=Lie(\bG)\ge \fp:=Lie(\bP)$. 
  \end{itemize}
  The top-row pullback $\fp$-module extension
  \begin{equation}\label{eq:plbk.ext}
    \begin{tikzpicture}[>=stealth,auto,baseline=(current  bounding  box.center)]
      \path[anchor=base] 
      (0,0) node (0l) {$0$}
      +(8,0) node (0r) {$0$,}
      +(1,0) node (p) {$\fp$}
      +(3,-.2) node (g) {$\fg$}
      +(5,-.2) node (uast) {$\fu^*$}
      +(7,-.3) node (gp) {$\fg/\fp$}
      +(2,.5) node (e) {$\fe'$}
      +(4,.5) node (past) {$\fp^*$}
      ;
      \draw[->] (0l) to[bend left=0] node[pos=.5,auto] {$\scriptstyle $} (p);
      \draw[->] (p) to[bend right=6] node[pos=.5,auto] {$\scriptstyle $} (g);
      \draw[->] (g) to[bend right=6] node[pos=.5,auto,swap] {$\scriptstyle $} (uast);
      \draw[->] (uast) to[bend right=6] node[pos=.5,auto] {$\scriptstyle \cong$} (gp);
      \draw[->] (g) to[bend right=20] node[pos=.5,auto,swap] {$\scriptstyle $} (gp);
      \draw[->] (p) to[bend left=6] node[pos=.5,auto] {$\scriptstyle $} (e);
      \draw[->] (e) to[bend right=10] node[pos=.5,auto] {$\scriptstyle $} (g);
      \draw[->>] (past) to[bend right=10] node[pos=.5,auto] {$\scriptstyle $} (uast);
      \draw[->] (e) to[bend left=6] node[pos=.5,auto] {$\scriptstyle $} (past);
      \draw[->] (past) to[bend left=6] node[pos=.5,auto] {$\scriptstyle $} (0r);
      \draw[->] (gp) to[bend right=6] node[pos=.5,auto] {$\scriptstyle $} (0r);
    \end{tikzpicture}
  \end{equation}
  with the double-arrow surjection dual to the embedding $\text{nilradical}=:\fu\lhook\joinrel\to \fp$ and the isomorphism implemented by $\Braket{-\mid -}$ is $\bK$-equivariantly isomorphic to that obtained via
  \begin{equation*}
    \Ext_{\fp}(\fp^*,\fp)
    \cong
    H^1\left(\fp,\Hom(\fp^*,\fp)\right)
    \cong
    H^1\left(\fp,\fp^{\otimes 2}\right)
    \in
    \cat{Rep}(\bK)
  \end{equation*}
  from the restriction to $\fp$ of the $(\fb,\fh,\Braket{-\mid-})$-standard bialgebra structure on $\fg$ for any $\bK$-invariant $\fh\le \fb\le \fp$. 
\end{theorem}

One (non-equivariant) consequence of this, \Cref{cor:fo.parab.sympl.lvs}, confirms the cursory remark on \cite[p.67]{FO98} identifying the symplectic leaves of the bialgebra-induced Poisson structure with the loci of bundles mutually isomorphic after forgetting the parabolic structure along the embedding $\bP\lhook\joinrel\to \bG$: a theme of some interest in several of the sources cited above (e.g. \cite[Theorem 6.6]{MR5018405}, \cite[Theorem 1.17]{MR4956460}).

\subsection*{Acknowledgments}

I am grateful for helpful exchanges with R. Kanda, M. Matviichuk and S. P. Smith.


\section{The comparison, and symplectic leaves as moduli-space map fibers}\label{se:coinc.pois}

We work over algebraically closed fields $\Bbbk$ of characteristic zero, with some of these strictures occasionally relaxed. 

\begin{remark}\label{re:rdctv.vs.ss}
  Some caution may be in order in perusing some of the cited sources, having to do with the distinction between semisimple and reductive Lie algebras (or algebraic groups):
  \begin{itemize}[wide]
  \item On the one hand, the informal discussion following \cite[Theorem 1]{FO98} begins by assuming $\bG$ semisimple, but the construction is intended as applicable, say, to $GL(n)$. 

  \item On the other hand, \cite[\S 2]{MR2396472} explicitly fixes a complex reductive $\bG$ and \cite[\S 3]{MR2396472} then makes the identification $\fg\cong \fg^*$ via the \emph{Killing form} of $\fg$; however, that form's non-degeneracy entails \cite[\S 5.1, Theorem]{hmph_1972} semisimplicity. Presumably, the construction must allow for other choices of invariant non-degenerate forms on $\fg$.
  \end{itemize}
\end{remark}

\Cref{con:bunp.poiss.no.bialg} summarizes the Poisson structure discussed in \cite[\S 3.1]{MR2396472}, while \Cref{con:bunp.poiss.bialg} recalls the slightly different (at least superficially) approach adopted in \cite[post Theorem 1, p.67]{FO98}. The two branches of the discussion share some initial data and conventions.

\begin{notation}\label{not:gp}
  \begin{enumerate}[(1),wide]
  \item\label{item:not:gp:gpb} $\bG$ will always denote a connected \emph{reductive} (\cite[\S 19.5]{hmph_lin-alg-gps_1981}, \cite[\S IV.11.21]{brl}) linear algebraic $\Bbbk$-group, with
    \begin{equation*}
      \bU
      :\xlongequal[\text{\cite[\S IV.11.21]{brl}}]{\text{\emph{unipotent radical}}}
      \cR_u\bP
      \le
      \bP
      \overset{\text{\emph{parabolic} \cite[\S 29.3, p.179]{hmph_lin-alg-gps_1981}}}{\le}
      \bG.
    \end{equation*}
    $\bB$ or $\bP_0$ stand for a \emph{Borel} subgroup contained in $\bP$. German letters denote corresponding Lie algebras: $\fg:=Lie(\bG)$, $\fu:=Lie(\bU)$, $\fb:=Lie(\bB)$ etc. ``-'' subscripts denote \emph{opposite} objects in the sense of \cite[\S IV.14.20]{brl}: $\bP_-$ is the unique \cite[Theorem IV.14.21]{brl} parabolic opposite $\bP$ containing a fixed \emph{Levi subgroup} $\bL\le \bP$, $\bU_-:=\cR_u\bP_-$, and similarly for Lie algebras. Finally, $\bH\le \bB$ is a maximal (algebraic) torus in $\bG$ with corresponding Lie algebra $\fh\le \fb$. 
    
    The various items will (mostly) carry compatible actions by a fixed linearly reductive affine group $\Bbbk$-scheme $\bK$, as in the statement of \Cref{th:same.poiss}.
    
  \item $E$ is an elliptic $\Bbbk$-curve, and given
    \begin{itemize}[wide]
    \item a principal $\bP$-bundle $\cE\xrightarrowdbl{} E$;
    \item and a $\bP$-representation $\rho:\bP\circlearrowright V$,
    \end{itemize}
    $\cE[V]:=\cE[\rho]$ both denote the associated \cite[Definition 5.3.1]{hjjm_bdle} vector bundle on $E$ based on that data. Unless specified otherwise, $\cE[\fp]$ refers to the adjoint representation $\bP\circlearrowright \fp$.

  \item Pairings $\Braket{-\mid -}$ have contextual meaning: the usual evaluation of a vector space against functionals thereon when the arguments reside in mutually dual vector spaces, or a fixed ad-invariant symmetric non-degenerate bilinear form on $\fg$ (one always exists, e.g. extending the Killing form on the semisimple derived subalgebra $\fg'$ with some freedom of choice on the center $\fz(\fg)$).      
  \end{enumerate}
\end{notation}

\begin{remark}\label{re:comods}
  Given the $\bK$-actions on $\bG$, $\bB$, etc., the Lie algebras involved will be internal to the category $\cM^H$ of $H$-comodules (\emph{closed symmetric monoidal} in the sense of \cite[Definition 6.1.3]{brcx_hndbk-2}) for the commutative \emph{cosemisimple} \cite[post Theorem 3.1.5]{dnr} Hopf algebra $H=\cO(\bK)$ of regular functions on $\bK$. $\fg$ thus carries an $H$-comodule structure
    \begin{equation*}
      \fg\ni x
      \xmapsto{\quad}
      x_0\otimes x_1\in \fg\otimes H
    \end{equation*}
    in \emph{Sweedler notation} \cite[\S 1.1.8]{dnr}, compatible with the Lie bracket in the guessable sense that that bracket is an $H$-comodule morphism $\fg^{\otimes 2}\to \fg$.
\end{remark}

$\bK$-equivariance will not pose any difficulties regarding $\Braket{-\mid -}$:

\begin{lemma}\label{le:exists.inv.bracket}
  If $\bK$ is a linearly reductive affine group $\Bbbk$-scheme acting on the connected reductive linear algebraic group $\bG$, $\fg:=Lie(\bG)$ carries a symmetric, non-degenerate, $\bK$-$\fg$-invariant bilinear form $\Braket{-\mid -}$. 
\end{lemma}
\begin{proof}
  $\fg$ splits \cite[\S 19.1, Proposition (a)]{hmph_1972} as the direct sum $\fg'\oplus \fz(\fg)$ of its derived subalgebra and center. On $\fg'$ the Killing form is non-degenerate and canonical, so $\bK$-invariant. Its kernel is precisely $\fz(\fg)$, so it remains to verify the claim for abelian $\bG$ (i.e. tori). In that case the $\bK$-action on $\fg$ factors through a finite subgroup of $\Aut(\bG)\cong GL\left(\dim_{\Bbbk} \bG,\bZ\right)$ (e.g. by \cite[Corollary III.8.3]{brl}). $\fg$-invariance is now no longer an issue; to conclude, observe that finite-group representations definable over $\bZ$ admit symmetric non-degenerate bilinear forms.
\end{proof}

$\bK$-invariance for the various objects involved perhaps requires some comment. Mostow's celebrated result \cite[first two Theorems, pp.200-201]{MR92928}, (essentially) to the effect that characteristic-0 linear algebraic groups have Levi factors, has equivariant analogues both in the algebraic-group context \cite[Proposition 6.1]{MR651417} and that of $H$-comodule Lie-algebras \cite[Theorem 7.4]{MR3115171}. One version, prescribing invariant Cartan data, is as follows. 

\begin{lemma}\label{le:k.inv.levi}
  Suppose the linearly reductive affine group $\Bbbk$-scheme $\bK$ acts on the connected reductive $\Bbbk$-group $\bG$ leaving a parabolic $\bP\le \bG$ invariant.
  \begin{enumerate}[(1),wide]
  \item\label{item:le:k.inv.levi:no.h} There is a $\bK$-invariant Levi factor $\bP\xrightarrowdbl{}\bL$ and a $\bK$-invariant $\bP$-opposite parabolic $\bP_{-}$.
  \item\label{item:le:k.inv.levi:h} If $\bK$ furthermore leaves invariant a Cartan/Borel/parabolic chain $\bH\le \bB\le \bP\le \bG$ then $\bL$ and $\bP_{-}$ above can be chosen so as to contain $\bH$. 
  \end{enumerate}
\end{lemma}
\begin{proof}
  \begin{enumerate}[label={},wide]
  \item\textbf{\Cref{item:le:k.inv.levi:no.h}} Invariant Levis are provided by \cite[Proposition 6.1]{MR651417}, and the second claim follows from the first: the \emph{unique} \cite[Proposition IV.14.21(i)]{brl} opposite parabolic containing the stipulated $\bK$-invariant Levi factor will (by uniqueness) be $\bK$-invariant.

  \item\textbf{\Cref{item:le:k.inv.levi:h}} Simply observe that Levi groups being precisely the centralizers of the tori maximal in the solvable radical $\cR\bP$ \cite[Proposition IV.11.23(ii)]{brl}, exactly one such contains the given $\bH$; it will then be $\bK$-invariant, whereupon proceed as in \Cref{item:le:k.inv.levi:no.h}. 
  \end{enumerate}
\end{proof}

\begin{construction}\label{con:bunp.poiss.no.bialg}
  \cite[\S 3.1]{MR2396472} proceeds to define a Poisson bivector on the smooth locus $\cat{Bun}_{E,s}^{\bP}$ in the following fashion.
  \begin{itemize}[wide]
  \item The tangent space to (the class of) $\left(\cE\xrightarrowdbl{}E\right)\in \cat{Bun}_{E,s}^{\bP}$ can be identified with $H^1(E,\cE[\fp])$, so the target is a skew-symmetric element of $H^1(E,\cE[\fp])^{\otimes 2}$; equivalently, a map
    \begin{equation}\label{eq:h1past.2.h1p}
      H^1\left(E,\cE[\fp]\right)^*
      \xrightarrow{\quad}
      H^1\left(E,\cE[\fp]\right)
    \end{equation}
    (with skew symmetry appropriately construed in that alternative description).

  \item In the first instance, the exact sequence
    \begin{equation*}
      0\to
      \fp
      \lhook\joinrel\xrightarrow{\quad}
      \fg
      \xrightarrowdbl{\quad}
      \fg/\fp
      \to 0
    \end{equation*}
    induces its analogue for vector bundles $\cE[\bullet]$, hence a morphism
    \begin{equation*}
      H^0(E,\cE[\fg/\fp])
      \xrightarrow{\quad}
      H^1(E,\cE[\fp])
    \end{equation*}
    resulting from the attached long exact cohomology sequence \cite[Theorem III.1.1A]{hrt}.

  \item \emph{Serre duality} \cite[Theorem III.7.1]{hrt} on $E$ identifies that map's domain with
    \begin{equation*}
      H^1\left(E,\cE[\fg/\fp]^*\right)^*
      \cong
      H^1\left(E,\cE\left[(\fg/\fp)^*\right]\right)^*.
    \end{equation*}

  \item This in turn produces a composition
    \begin{equation*}
      H^1\left(E,\cE[\fg/\fu]^*\right)^*
      \xrightarrow{\quad\text{natural map}\quad}
      H^1\left(E,\cE[\fg/\fp]^*\right)^*
      \xrightarrow{\quad}
      H^1(E,\cE[\fp]).
    \end{equation*}

  \item Finally,
    \begin{equation*}
      \left(\fg/\fu\right)^*
      \cong
      \fp
      \xRightarrow{\quad}
      \cE\left[\left(\fg/\fu\right)^*\right]
      \cong
      \cE[\fp]
    \end{equation*}
    via $\Braket{-\mid-}$. 
  \end{itemize}
\end{construction}

On the other hand:

\begin{construction}\label{con:bunp.poiss.bialg}
  The discussion on \cite[p.67]{FO98} takes a slightly different route to defining \Cref{eq:h1past.2.h1p}: it is rather
  \begin{equation*}
    H^1\left(E,\cE\left[\fp\right]\right)^*
    \xrightarrow[\quad\cong\quad]{\quad\text{Serre duality}\quad}
    H^0\left(E,\cE\left[\fp^*\right]\right)
    \xrightarrow[\quad\quad]{\quad\text{long exact sequence}\quad}
    H^1\left(E,\cE[\fp]\right)
  \end{equation*}
  with the second map resulting via $\cE[\bullet]$ from an extension
  \begin{equation}\label{eq:past.by.p.via.bialg}
    0\to
    \fp
    \lhook\joinrel\xrightarrow{\quad}
    \fe
    \xrightarrowdbl{\quad}
    \fp^*
    \to 0
  \end{equation}
  of $\fp$-modules built via Lie-algebra cohomology. Specifically, the class of \Cref{eq:past.by.p.via.bialg} in
    \begin{equation*}
      \Ext_{\fp}\left(\fp^*,\fp\right)
      \xrightarrow[\quad\cong\quad]{\quad}
      H^1\left(\fp,\fp^{\otimes 2}\right)
      \quad
      \left(\text{\cite[Exercise 7.3.5]{weib_halg}}\right)
    \end{equation*}
    is represented by a 1-cocycle $\fp\to \wedge^2 \fp\le \fp^{\otimes 2}$ (exterior square) that the source in question refers to as being (the restriction to $\fp$ of) that inducing the \emph{standard bialgebra structure} $\fg\to \wedge^2 \fg$.

    What the authors intend is presumably something along the lines of \cite[Example 1.3.8]{cp_qg}, describing such standard constructions for symmetrizable Kac-Moody algebras. Loc. cit. reveals the requisite ingredients for the construction to go through:
    \begin{itemize}[wide]
    \item an ad-invariant non-degenerate symmetric bilinear form $\Braket{-\mid -}$ on $\fg$; per the convention adopted here, this will be the only such in place throughout and hence coincide with that employed in \Cref{con:bunp.poiss.no.bialg};

    \item and a Borel subalgebra $\fh\le \fb\le \fp$ with a Cartan algebra therein.
    \end{itemize}
\end{construction}

The crux of the matter in comparing the two approaches, then, lies in the contrast between two $\fp$-module extensions \Cref{eq:plbk.ext} and (the bialgebra-induced) \Cref{eq:past.by.p.via.bialg}.

\begin{remarks}\label{res:derived.duality}
  \begin{enumerate}[(1),wide]
  \item\label{item:res:derived.duality:drv.dual} The cited \cite[Exercise 7.3.5]{weib_halg}, providing the identification $\Ext_{\fg}(M,N)\cong H^1(\fg,\Hom(M,N))$ for modules over a Lie algebra $\fg$, is a particular instance of the following observation (also of some use in bookkeeping some of the connections between various Poisson-structure descriptions). 
    
    Consider a closed symmetric monoidal $(\cC,\otimes,\mathbf{1})$, linear over a field $\Bbbk$, with \emph{internal homs} $[M,N]$. Under appropriate exactness constraints on $(\cC,\otimes,\mathbf{1})$ (valid in familiar examples such as $\Bbbk$-linear categories of modules over groups or Lie algebras, or more generally (co)module categories over Hopf $\Bbbk$-algebras) one can pass back and forth between extensions
    \begin{equation*}
      \cE\quad:\quad
      0\to
      N
      \lhook\joinrel\xrightarrow{\quad}
      E
      \xrightarrowdbl{\quad}
      M
      \to 0
    \end{equation*}
    and
    \begin{equation*}
      \cF\quad:\quad
      0\to
      [M,N]
      \lhook\joinrel\xrightarrow{\quad}
      F
      \xrightarrowdbl{\quad}
      \mathbf{1}
      \to 0
    \end{equation*}
    (identified up to the usual \cite[\S 3.4]{weib_halg} notion of equivalence).  
    \begin{itemize}[wide]
    \item Given $\cE$, take for $F$ the pullback
      \begin{equation}\label{eq:f1ennn}
        \begin{tikzpicture}[>=stealth,auto,baseline=(current  bounding  box.center)]
          \path[anchor=base] 
          (0,0) node (l) {$F$}
          +(2,.5) node (u) {$[E,N]$}
          +(2,-.5) node (d) {$\mathbf{1}$}
          +(4,0) node (r) {$[N,N]$}
          ;
          \draw[->] (l) to[bend left=6] node[pos=.5,auto] {$\scriptstyle $} (u);
          \draw[->] (u) to[bend left=6] node[pos=.5,auto] {$\scriptstyle [-,N]$} (r);
          \draw[->] (l) to[bend right=6] node[pos=.5,auto,swap] {$\scriptstyle $} (d);
          \draw[->] (d) to[bend right=6] node[pos=.5,auto,swap] {$\scriptstyle $} (r);
        \end{tikzpicture}
      \end{equation}
      with the bottom right-hand map being that attached by the adjunction $-\otimes N\dashv [N,-]$ to $\id_N$ (``$\dashv$'''s tail pointing towards the left adjoint, in not-uncommon notation \cite[Definition 19.3]{ahs}). 

    \item Conversely, $\cF$ in hand, take for $\cE$ the pushout
      \begin{equation*}
        \begin{tikzpicture}[>=stealth,auto,baseline=(current  bounding  box.center)]
          \path[anchor=base] 
          (0,0) node (0l) {$0$}
          +(2,.5) node (ul) {$[M,N]\otimes M$}
          +(4.5,.7) node (um) {$F\otimes M$}
          +(6,0) node (r) {$M$}
          +(2,-.5) node (dl) {$N$}
          +(4.5,-.7) node (dm) {$E$}
          +(7,0) node (0r) {$0$}
          ;
          \draw[->] (0l) to[bend left=6] node[pos=.5,auto] {$\scriptstyle $} (ul);
          \draw[->] (0l) to[bend right=6] node[pos=.5,auto] {$\scriptstyle $} (dl);
          \draw[->] (r) to[bend right=0] node[pos=.5,auto] {$\scriptstyle $} (0r);
          \draw[->] (um) to[bend left=6] node[pos=.5,auto] {$\scriptstyle $} (r);
          \draw[->] (dm) to[bend right=6] node[pos=.5,auto] {$\scriptstyle $} (r);
          \draw[->] (ul) to[bend left=0] node[pos=.5,auto] {$\scriptstyle $} (dl);
          \draw[->] (um) to[bend left=0] node[pos=.5,auto] {$\scriptstyle $} (dm);
          \draw[->] (ul) to[bend left=6] node[pos=.5,auto] {$\scriptstyle $} (um);
          \draw[->] (dl) to[bend right=6] node[pos=.5,auto] {$\scriptstyle $} (dm);
        \end{tikzpicture}
      \end{equation*}
      of $\cF\otimes M$ along the left-hand vertical map (canonical, internal evaluation: the $M$-component of the counit of the adjunction $-\otimes M\dashv [M,-]$).
    \end{itemize}
    
    This all applies in particular when $\cC$ is \emph{rigid} \cite[Definition 2.10.11]{egno} (so that $[M,-]\cong -\otimes M^*$) with exact tensor products $M\otimes-$ and $-\otimes N$, finite-dimensional $\fg$- or $\bG$-modules as the familiar examples.

  \item\label{item:res:derived.duality:which.conds} It will be clear from \Cref{item:res:derived.duality:drv.dual} above what the nature of the exactness requirements on $\cC$ alluded to must be:
    \begin{itemize}[wide]
    \item for the passage $\cE \rightsquigarrow \cF$ one would need $[-,N]$ exact (so that the upper right-hand arrow in \Cref{eq:f1ennn} is epic) as well as $\mathbf{1}\to [N,N]$ monic;
    \item in $\cF\rightsquigarrow \cE$, on the other hand, the exactness of $-\otimes M$ is taken for granted.
    \end{itemize}
    As observed, these are automatic in the usual representation categories (e.g. over Hopf algebras) over fields. There is a dual approach to $\cE\rightsquigarrow \cF$: one can take for $F$ the pullback
    \begin{equation*}
      \begin{tikzpicture}[>=stealth,auto,baseline=(current  bounding  box.center)]
        \path[anchor=base] 
        (0,0) node (l) {$F$}
        +(2,.5) node (u) {$[M,E]$}
        +(2,-.5) node (d) {$\mathbf{1}$}
        +(4,0) node (r) {$[M,M]$;}
        ;
        \draw[->] (l) to[bend left=6] node[pos=.5,auto] {$\scriptstyle $} (u);
        \draw[->] (u) to[bend left=6] node[pos=.5,auto] {$\scriptstyle [M,-]$} (r);
        \draw[->] (l) to[bend right=6] node[pos=.5,auto,swap] {$\scriptstyle $} (d);
        \draw[->] (d) to[bend right=6] node[pos=.5,auto,swap] {$\scriptstyle $} (r);
      \end{tikzpicture}
    \end{equation*}
    it is now the exactness of $[M,-]$ that is needed.

  \item\label{item:res:derived.duality:kom.adj} Exactness aside, \Cref{item:res:derived.duality:drv.dual} does provide an adjunction between two appropriately defined categories:
    \begin{equation*}
      \begin{tikzpicture}[>=stealth,auto,baseline=(current  bounding  box.center)]
        \path[anchor=base] 
        (0,0) node (l) {$\cat{Kom}(\mathbf{1},[M,N])$}
        +(6,0) node (r) {$\cat{Kom}(M,N)$,}
        +(3,0) node () {$\bot$}
        ;
        \draw[->] (l.north east) to[bend left=16] node[pos=.5,auto] {$\scriptstyle \cF\mapsto \cE_{\cF}:=\cE$} (r.north west);
        \draw[->] (r.south west) to[bend left=16] node[pos=.5,auto] {$\scriptstyle \cF=:\cF_{\cE}\mapsfrom \cE$} (l.south east);
      \end{tikzpicture}
    \end{equation*}
    having defined 
    \begin{equation*}
      X,Y\in \cC\quad:\quad
      \cat{Kom}(X,Y)
      :=
      \left\{\text{complexes }Y\to E\to X\right\},
    \end{equation*}
    with complex morphisms restricting to $\id_X$ and $\id_Y$ on the extreme terms for arrows.
  \end{enumerate}
\end{remarks}

It will be instructive to analyze the two extreme cases.

\begin{example}\label{ex:pg}
  For $\bP:=\bG$ the top row in \Cref{eq:plbk.ext} will be a trivial extension of $\fg$-modules, for it factors through $\fu^*=\{0\}$. On the other hand, \Cref{eq:past.by.p.via.bialg} is the extension of $\fg^*$ by $\fg$ employed in \cite[Example 1.3.8]{cp_qg} in embedding $\fg$ into the \emph{Manin triple} \cite[Definition 1.3.3]{cp_qg} structure
  \begin{equation*}
    \begin{aligned}
      (\fa,\fa_0,\fa_1)
      &:=\left(\fg\times \fg,\Delta\fg,\fq\right)
        \cong\left(\fg\times \fg,\fg,\fg^*\right)
      \\
      \Delta&:=\text{diagonal map}\\
      \fq&:=\left\{(a,b)\in \fb_-\times \fb\ :\ a+b\text{ has 0 component in the Cartan }\fh\le \fb\right\}:
    \end{aligned}
  \end{equation*}
  this last claim follows from \Cref{le:bialg.as.1cocyc}, given that the Lie subalgebra structure on $\fq\le \fa$ and the identification $\fq\cong \fg^*$ via the invariant inner product
  \begin{equation*}
    \Braket{(a,b)\mid (a',b')}_{\fa}
    :=
    \Braket{a\mid a'}-\Braket{b\mid b'}
  \end{equation*}
  on $\fa$ are precisely those used to make $\fg^*$ into a Lie algebra and hence give $\fg$ its standard Lie coproduct $\fg\xrightarrow{\delta}\wedge^2 \fg\le \fg^{\otimes 2}$. $\fa$ indeed splits as an extension of $\fg$-modules (so we do have \Cref{eq:past.by.p.via.bialg} $\cong$ \Cref{eq:plbk.ext} in this case): the \emph{anti}diagonal $\{(a,-a)\ :\ a\in \fg\}\le \fa$ is a $\fg$-module supplementing $\Delta\fg$. 
\end{example}

Recall \cite[post Theorem 4.6]{grv_cls-gp-geom-alg_2002} that an \emph{isotropic} subspace $W\le V$ of a space equipped with a bilinear form $\Braket{-\mid -}$ is one with $\Braket{W\mid W}\equiv 0$ and that the \emph{classical double} \cite[\S 1.4B]{cp_qg}
\begin{equation}\label{eq:cls.dbl}
  0\to
  \fg
  \lhook\joinrel\xrightarrow{\quad}
  \cD(\fg)
  \xrightarrowdbl{\quad}
  \fg^*
  \to 0
\end{equation}
of the finite-dimensional Lie bialgebra $(\fg,\delta)$ over a field $\Bbbk$ (of characteristic $\ne 2$, to avoid symmetric/exterior-square issues) for a 1-cocycle
\begin{equation}\label{eq:curried.delta}
  \delta
  \in
  \Hom\left(\fg,\fg^{\otimes 2}\right)
  \xrightarrow[\quad\cong\quad]{\quad\text{dualize second tensorand}\quad}
  \Hom\left(\fg,\Hom\left(\fg^*,\fg\right)\right)
\end{equation}
is
\begin{itemize}[wide]
\item $\cD(\fg):=\fg\times \fg^*$ with the summands $\fg,\fg^*\le \cD(\fg)$ embedded as Lie subalgebras;
\item with $\cD(\fg)$ equipped with the unique Lie-algebra structure admitting such embeddings and leaving the symmetric non-degenerate bilinear form
  \begin{equation*}
    \Braket{-\mid-}
    \in
    \Hom\left(S^2 \cD(\fg),\Bbbk\right)
    \le
    \Hom\left(\cD(\fg)^{\otimes 2},\Bbbk\right)
    ,\quad
    \left[
      \begin{gathered}
        \Braket{-\mid -}|_{\fg\times \fg^*\oplus \fg^*\times \fg}:=\text{natural pairing}\\
        \fg,\ \fg^*\ \le\  \cD(\fg)\ \text{$\Braket{-\mid-}$-isotropic}\\
        S^2\bullet:=\text{symmetric square}
      \end{gathered}
    \right.
  \end{equation*}
  invariant. The construction functions without the finite dimensionality assumption, regarding $\delta$ as an element of the right-hand side of \Cref{eq:curried.delta}. 
\end{itemize}
\Cref{le:bialg.as.1cocyc}, taken for granted in \Cref{ex:pg}, is presumably folklore; an argument is included for completeness (cf. the direct computational \cite[proof of Lemma 1.3.5]{cp_qg} via dual bases in the finite-dimensional case and, in the same spirit, \cite[Remark 2.15]{MR4214399}). 

\begin{lemma}\label{le:bialg.as.1cocyc}
  Let $\left(\fg,\delta\right)$ be a Lie $\Bbbk$-bialgebra, $\mathrm{char}(\Bbbk)\ne 2$. The extension of $\fg^*$ by $\fg$ associated to $\delta$ through
  \begin{equation*}
    \Ext^1_{\fg}(\fg^*,\fg)
    \cong
    H^1(\fg,\Hom(\fg^*,\fg))
    \left(      
      \cong
      H^1\left(\fg,\fg^{\otimes 2}\right)
      \text{ if }
      \dim\fg<\infty
    \right)
  \end{equation*}
  is precisely the classical double \Cref{eq:cls.dbl}.
\end{lemma}
\begin{proof}

  Having fixed the defining splitting $\cD(\fg):=\fg\times \fg^*$, a representing 1-cocycle is given by
  \begin{equation*}
    \fg\ni
    x\xmapsto{\quad}
    \left(
      \fg^*\ni y
      \xmapsto{\quad}
      \left[x,y\right]_{\fg}
      \in \fg\le \cD(\fg)
    \right),\quad
    \bullet_{\fg}:=\text{$\fg$-summand}. 
  \end{equation*}
  That this recovers \Cref{eq:curried.delta} is virtually tautological:
  \begin{equation*}
    \Braket{\left[x,y\right]_{\fg}\mid z\in \fg^*}
    \xlongequal[\quad\Braket{-\mid-}\text{ invariance}\quad]{\quad}
    \Braket{x\mid [y,z]}
    =
    \Braket{\delta(x)(y)\mid z},
  \end{equation*}
  hence the conclusion.
\end{proof}

\begin{example}\label{ex:pb}  
  At the other end of the spectrum (as compared to \Cref{ex:pg}), consider the minimal choice $\fp:=\fp_0=\fb$. We can again identify \Cref{eq:past.by.p.via.bialg} and \Cref{eq:plbk.ext}, by producing coincident classifying cocycles. 

  To describe one cocycle whose class in $H^1_{\fb}(\fb^*,\fb)$ classifies \Cref{eq:past.by.p.via.bialg}:
  \begin{itemize}[wide]
  \item in \Cref{ex:pg}'s Manin triple $(\fa,\fa_0,\fa_1)=\left(\fg^2,\Delta\fg,\fq\right)$, operate with $x\in \fb\cong \Delta\fb$ on elements $y\in \fq\cong \fg^*$;
  \item note that that action in fact factors through $\fg^*\xrightarrowdbl{}\fb^*$, for the annihilator of $\fb\cong \Delta\fb$ in $\fq$ with respect to $\Braket{-\mid -}_{\fa}$ is an ideal in $\fq$;
  \item evaluate the $\left(\fg\cong \Delta\fg\right)$-component $z$ of the result $[x,y]\in \fa$ in the decomposition $\fa=\Delta\fg\oplus \fq$;
  \item that component will lie in $\fb$. 
  \end{itemize}
  Because the identification $\fq\cong \fg^*$ further descends to
  \begin{equation*}
    \fb^*
    \cong
    \fq/\left\{(0,z)\ :\ z\in \fu\right\}
    \cong
    \fb_{-}
    \cong
    \left\{(x,y)\in \fb_{-}\times \fh\ :\ x+y\in \fu_{-}\right\},
  \end{equation*}
  the description just given is also how one would produce a cocycle classifying the top-row extension \Cref{eq:plbk.ext} (having first split $\fg$, as a $\fb$-module extension of $\fu^*$ by $\fb$, by identifying $\fu^*$ with $\fu_{-}$ via $\Braket{-\mid -}$):
  \begin{equation*}
    \fb\ni x
    \xmapsto{\quad}
    \left(
      \fb^*\cong \fb_{-}\ni y
      \xmapsto{\quad}
      \left[x,y_{\fu_{-}}\right]_{\fb}
      \in \fb
    \right),
  \end{equation*}
  where subscripts denote components in the decomposition $\fg=\fu\oplus \fh\oplus \fu_{-}=\fb\oplus \fu_{-}$. 
\end{example}

\Cref{ex:pb} extends to confirm that the two Poisson-structure constructions do indeed coincide regardless of the choice of parabolic group $\bP$. 

\pf{th:same.poiss}
\begin{th:same.poiss}
  As indicated preceding the statement, the argument is an adaptation of \Cref{ex:pb}. The auxiliary objects featuring in the proof, such as the Levi factor $\bL\le \bP$ and the opposite parabolic $\bP_{-}\ge \bL$, will all be $\bK$ invariant because chosen as \Cref{le:k.inv.levi} permits; this applies also to downstream notions (such as the Lie algebras $\fq_{\fp,\fl}$).
  
  A classifying cocycle for \Cref{eq:past.by.p.via.bialg}, being obtained by restriction from one over $\fg$ classifying \Cref{eq:cls.dbl}, will be cohomologous to 
  \begin{equation}\label{eq:1cocyc.dscr}    
    \left(\Delta\fp\le \fa\right)
    \cong
    \fp\ni
    x
    \xmapsto{\quad}
    \left(
      \fg^*\cong \ker\pi\ni
      y
      \xmapsto{\quad}
      \pi[x,y]
      \in \fg
    \right),
  \end{equation}
  where
  \begin{equation*}
    \cD(\fg)
    \xrightarrow[\quad\cong\quad]{\quad\text{isomorphism of $\fg^*$-by-$\fg$ extensions}\quad}
    \cong
    \fa
    \xrightarrowdbl{\quad\pi\quad}
    \fg
  \end{equation*}
  is a (vector-space, not $\fp$- or $\fg$-module) splitting of $\fg\lhook\joinrel\to \cD(\fg)\cong \fa$. A judicious choice of splitting will help: interpolating between the anti-diagonal copy of $\fg$ chosen in \Cref{ex:pg} and $\fq$ itself, driving \Cref{ex:pb}, set
  \begin{equation*}
    \ker\pi:=\fq_{\fp,\fl}
    :=
    \left\{(x,y)\in \fp_-\times \fp\ :\ x+y\in \fu_{-}\oplus \fu\le \fu_{-}\oplus \fl\oplus \fu\right\}
  \end{equation*}
  for the fixed Levi subalgebra $\fl:=Lie(\bL)\le \fp$, so that $\fq_{\fg,\fg}$ is the anti-diagonal $\fg\le \fa$ and $\fq_{\fb,\fh}=\fq$.   

  The annihilator of $\fp\cong \Delta\fp\le \fa$ in $\fq_{\fp,\fl}$ is $\{(0,y)\ :\ y\in \fu\}$, so that $\fg^*\cong \fq_{\fp,\fl}$ induces
  \begin{equation*}
    \fp^*
    \cong
    \fq_{\fp,\fl}/\left\{(0,y)\ :\ y\in \fu\right\}
    \cong
    \left\{(x,y)\in \fp_{-}\times \fl\ :\ x+y\in \fu_{-}\right\}
    \cong
    \fp_{-}. 
  \end{equation*}
  \Cref{eq:1cocyc.dscr} can thus be recovered as
  \begin{equation}\label{eq:alt.desc.1cocyc}
    \fp\ni
    x
    \xmapsto{\quad}
    \left(
      \fp^*
      \cong
      \fp_{-}
      \ni
      y
      \xmapsto{\quad}
      \left[x,y_{\fu_{-}}\right]_{\fp}
      \in \fp
    \right),
  \end{equation}
  with subscripts denoting the respective components in $\fg=\fu\oplus \fl\oplus \fu_{-}=\fp\oplus \fu_{-}$. Splitting \Cref{eq:plbk.ext}'s $\fg\xrightarrowdbl{}\fu^*$ by identifying the latter with $\fu_{-}$, \Cref{eq:alt.desc.1cocyc} plainly classifies the pullback extension $\fe'$.
\end{th:same.poiss}

An immediate consequence of \Cref{th:same.poiss}:

\begin{corollary}\label{cor:hence.ell.curve}
  If in addition $E$ is an elliptic $\Bbbk$-curve, \Cref{con:bunp.poiss.no.bialg} and \Cref{con:bunp.poiss.bialg} describe the same Poisson structure on the smooth locus $\cat{Bun}_{E,s}^{\bP}$.  \qedhere
\end{corollary}

\Cref{th:same.poiss} affords one verification of the claim made in passing in \cite[discussion post Theorem 1, p.67]{FO98} that the symplectic leaves of the resulting Poisson structure on $\cat{Bun}_{E,s}^{\bP}$ are precisely the loci of $\bP$-bundles mutually isomorphic as $\bG$-bundles. 

\begin{corollary}\label{cor:fo.parab.sympl.lvs}
  The symplectic leaves of \Cref{con:bunp.poiss.bialg}'s Poisson structure on $\cat{Bun}_{E,s}^{\bP}$ are the fibers of $\cat{Bun}_{E,s}^{\bP}\to \cat{Bun}_{E}^{\bG}$. 
\end{corollary}
\begin{proof}
  This is so for \Cref{con:bunp.poiss.no.bialg} instead by \cite[Corollary 3.7]{MR2396472}: that result applies to the fibers of structure-group compression the map $\cat{Bun}^{\bP}_{E,s}\xrightarrow{\bullet/\bU}\cat{Bun}^{\bL}_E$, but those fibers are proven Poisson submanifolds in \cite[Lemma 3.4]{MR2396472}, so \emph{their} symplectic leaves are global leaves \cite[Proposition 8.2]{cfm_lec-poiss_2021}. The conclusion follows from \Cref{th:same.poiss}.
\end{proof}


\addcontentsline{toc}{section}{References}

\def\polhk#1{\setbox0=\hbox{#1}{\ooalign{\hidewidth
  \lower1.5ex\hbox{`}\hidewidth\crcr\unhbox0}}}


\Addresses

\end{document}